%% file: KashinDecompCS.tex
\documentclass[11pt,reqno]{amsart}
\usepackage{fullpage,times,graphicx,amssymb,amsmath,psfrag,xcolor}
\definecolor{ddarkbrown}{rgb}{0.5,0.2,0.05} \definecolor{bbluegray}{rgb}{0.05,0,0.5}
\usepackage[colorlinks,citecolor=bbluegray,linkcolor=ddarkbrown,urlcolor=blue,breaklinks]{hyperref}
\usepackage[off]{auto-pst-pdf} 
\usepackage[round]{natbib}
\input defs.tex

\begin{document}
\title{Sparse Recovery, Kashin Decomposition and Conic Programming}
\author{Alexandre d'Aspremont}
\address{ORFE, Princeton University, Princeton, NJ 08544.}
\email{aspremon@princeton.edu}

\keywords{Compressed Sensing, Kashin Decomposition, Semidefinite Programming.}
\date{Jan. 12 2011}
\subjclass[2010]{94A12, 90C27, 90C22}

\begin{abstract}
We produce relaxation bounds on the diameter of arbitrary sections of the $\ell_1$ ball in $\reals^n$. We use these results to test conditions for sparse recovery.
\end{abstract}
\maketitle

\section{Introduction}
Let $A\in\reals^{m \times n}$ be a full rank matrix, we are given $m$ observations $Au$ of a signal $u\in\reals^n$, and we seek to decode it by solving
\BEQ\label{eq:l0-dec}
\BA{ll}
\mbox{minimize} & \Card(x)\\
\mbox{subject to} & Ax=Au,
\EA
\EEQ
in the variable $x\in\reals^n$. Problem~\eqref{eq:l0-dec} is combinatorially hard, but under certain conditions on the matrix $A$ (see e.g. \citet{Dono05,Cand05,Kash07,Cohe06}), we can reconstruct the signal by solving instead
\BEQ\label{eq:l1-dec}
\BA{ll}
\mbox{minimize} & \|x\|_1\\
\mbox{subject to} & Ax=Au,
\EA
\EEQ
which is a convex problem in the variable $x\in\reals^n$.

\section{Sparse recovery conditions}
We begin by discussing conditions on the coding matrix $A\in\reals^{m \times n}$ and on the signal $u$ which guarantee that the solution to the $\ell_1$ minimization problem~\eqref{eq:l1-dec} matches that of the $\ell_0$ minimization problem~\eqref{eq:l0-dec} and allows us to reconstruct the original signal $u$.

\subsection{Discrete signals}
We first assume that the signal $u$ only takes discrete values. For a given coding matrix $A\in\reals^{m \times n}$, the proposition below describes a sufficient condition which guarantees that a discrete signal $u\in\{-1,0,1\}^n$ will be reconstructed by solving problem~\eqref{eq:l1-dec}.

\begin{proposition} \label{prop:l1-rec}
We define
\BEQ\label{eq:u-set}
{\mathcal U}=\left\{u\in\reals^n:~u^Tx + \xi \sum_{i=1}^n |u_i x_i| \leq \xi \|x\|_1, \forall x \in \reals^n:Ax=0 \right\}.
\EEQ
If $u\in\{-1,0,1\}^n \cap {\mathcal U}$ for some $\xi\in(0,1)$ and $z\in\reals^n$ solves the $\ell_1$ recovery problem in~\eqref{eq:l1-dec}, then the signature of $z$ is a subset of that of $u$, i.e. $u_iz_i=|z_i|$, $i=1,\ldots,n$.
\end{proposition}
\begin{proof}
Suppose there is a vector $z\in\reals^n$, with $Az=Au$ and $\|z\|_1\leq\|u\|_1$. Let $I=\{i\in[1,n]:u_i\neq 0\}$ be the support of the signal $u$ and $J$ its complement in $[1,n]$, the vector $u-z$ is in the nullspace of~$A$ so $u\in{\mathcal U}$ implies
\[
u^T(u-z) + \xi \sum_{i=1}^n |u_i| |u_i-z_i| \leq \xi \|u-z\|_1.
\]
Because $u_i\in\{-1,0,1\}$ this is equivalent to
\[
u^T(u-z) \leq \xi \|z_J\|_1,
\]
hence, having assumed $\|z\|_1\leq\|u\|_1$, we get
\BEAS
\|z\|_1 ~ = ~ \|z_I\|_1 + \|z_J\|_1 ~ \leq  ~ \|u\|_1 &\leq & u^Tz + \xi \|z_J\|_1\\
&\leq & \|z_I\|_1 + \xi \|z_J\|_1,
\EEAS
so $\|z_J\|_1=0$. Then $\|z\|_1 \leq \|u\|_1 \leq u^Tz \leq \|z\|_1$ means $u_i z_i=|z_i|$, $i=1,\ldots,n$.
\end{proof}

Given a priori bounds on the signal coefficients, we obtain the following (tighter) result, which ensures that the signature of the decoded signal matches that of the true one, when solving a modified version of problem~\eqref{eq:l1-dec}.
\begin{corollary}\label{prop:l1-rec-cube}
Let z solve
\BEQ\label{eq:l1-dec-cube}
\BA{ll}
\mbox{minimize} & \|x\|_1\\
\mbox{subject to} & Ax=Au\\
& \|x\|_\infty \leq 1,
\EA\EEQ
If $u\in\{-1,0,1\}^n \cap {\mathcal U}$ for some $\xi\in(0,1)$, where ${\mathcal U}$ was defined in~\eqref{eq:u-set}, then $z=u$.
\end{corollary}
\begin{proof}
In the proof of Proposition~\ref{prop:l1-rec}, we showed $u^Tz=\|z\|_1=\|u\|_1$ under the same assumptions, which together with the additional constraint that $\|z\|_\infty \leq 1$ means that $z=u$.
\end{proof}

Next, we show that controlling the ratio of dual pairs of norms on the nullspace of $A$ provides simple sufficient conditions for checking that a signal $u$ belongs to the set ${\mathcal U}$ of $\ell_1$-recoverable signals.

\begin{proposition} \label{prop:univ-cond}
Let $\|\cdot\|$ be a norm on $\reals^n$ and $\|\cdot\|_*$ its dual, $A\in\reals^{m \times n}$ and $u\in\{-1,0,1\}^n$, if
\BEQ\label{eq:univ-cond}
\sup_{\substack{Ax=0,\\ \|x\|_1\leq 1}} \|x\| < \frac{1}{\|u\|_*}
\EEQ
then $u\in{\mathcal U}$, where ${\mathcal U}$ is the set of $\ell_1$-recoverable signals defined in~\eqref{eq:u-set}.
\end{proposition}
\begin{proof}
When~\eqref{eq:univ-cond} holds
\[
\sup_{\substack{Ax=0\\ \|x\|_1\leq 1}}~|u|^T|x|\leq \|u\|_*\sup_{\substack{Ax=0,\\ \|x\|_1\leq 1}} \|x\| \leq \frac {\xi}{1+\xi}
\]
for some $\xi\in(0,1)$, where $|u|$ is the vector with components $|u_i|$. We then have
\[
\sup_{\substack{Ax=0, \|x\|_1\leq 1\\ w\in\{-1,1\}^n}}~u^T(\idm+\xi \diag(w))x
\leq
(1+\xi)\sup_{\substack{Ax=0\\ \|x\|_1\leq 1}}~|u|^T|x| \leq \xi
\]
which means $u\in{\mathcal U}$.
\end{proof}

We can bound the value of the $\sup$ in~\eqref{eq:univ-cond} when $\|\cdot\|$ is the Euclidean norm and the matrix $A$ satisfies the restricted isometry property of order $k^*$ with constant $\delta<1$.

\begin{lemma}\label{lem:nrm-ratio-RIP}
Suppose $A\in\reals^{m \times n}$ satisfies the Restricted Isometry Property (RIP) of order $3k^*$ with constant~$\delta_{3k^*}\leq \delta < 1$, then
\BEQ\label{eq:ratio-bnd}
\sup_{Ax=0} \frac{\|x\|_2}{\|x\|_1}~ \leq ~ \frac{2}{(1-\delta)\sqrt{k^*}}.
\EEQ
\end{lemma}
\begin{proof}
We roughly follow the proof of \citet[Lemma 4.1]{Cohe06}. Let $\eta\in\reals^n$ be in the nullspace of~$A$. Let $T=T_0$ be the index set of the $k^*$ largest magnitude coefficients in $\eta$, with $T_1$ corresponding to the next $k^*$ largest and so on. \citet[Lemma 4.1]{Cohe06} show
\[
\|\eta_T\|_2 \leq \frac{(1+\delta)}{(1-\delta)}\frac{\|\eta_{T^c}\|_1}{\sqrt{k^*}}
\]
and
\[
\|\eta_{T_{i+1}}\|_2 \leq \frac{\|\eta_{T_i}\|_1}{\sqrt{k^*}}, \quad i\geq 0
\]
so 
\[
\|\eta_{T^c}\|_2 \leq \frac{\|\eta\|_1}{\sqrt{k^*}}
\]
and a triangular inequality yields the desired result.
\end{proof}

As discussed in \citep{Dono06a,Kash07} this result is in fact a direct consequence of classical bounds on Gel'fand and Kolmogorov widths, with \cite{Kash77,Garn84} showing in particular that
\[
\sup_{Ax=0} \frac{\|x\|_2}{\|x\|_1} \leq \frac{8}{\sqrt{n}}
\]
for some matrices $A$ (the proof is not constructive). Moreover, \cite{Kash77} shows that this holds with high probability when the nullspace of $A$ is picked at random uniformly on the Grassman manifold of subspaces of $\reals^n$ with dimension~$k\leq n/2$. In other words, when $n$ is large, most matrices are good sensing matrices.

\subsection{Generic signals}
Very similar results hold for arbitrary signals $u\in\reals^n$ at marginally lower thresholds. In particular, \citet[Th. 2.1]{Kash07} show the following guarantee.
\begin{proposition}\label{prop:card-S}
Given a coding matrix $A\in\reals^{m \times n}$, suppose that there is some $S>0$ such that
\BEQ\label{eq:diam}
\sup_{Ax=0} \frac{\|x\|_2}{\|x\|_1} \leq \frac{1}{\sqrt{S}}
\EEQ
then $x^\mathrm{LP}=u$ if $\Card(u)\leq S/4$, and
\[
\|u-x^\mathrm{LP}\|_1 \leq ~4\min_{\{\Card(y)\leq S/16\}} ~\|u-y\|_1
\]
where $x^\mathrm{LP}$ solves the $\ell_1$-recovery problem in~\eqref{eq:l1-dec} and $u$ is the original signal.
\end{proposition}
This means that the $\ell_1$-minimization problem in~\eqref{eq:l1-dec} will recover exactly all sparse signals $u$ satisfying $\Card(u)\leq S/4$ and that the $\ell_1$ reconstruction error for other signals will be at most four times larger than the $\ell_1$ error corresponding to the best possible approximation of $u$ by a signal of cardinality at most $S/16$.

\section{Weak recovery conditions}\label{s:weak}
Similar conditions (with slightly better recovery thresholds) can be derived when the signal $u$ follows a given distribution, and recovery is only required to occur with high probability. Given $k\in[0,n]$, suppose now that the signal is i.i.d., distributed as follows
\BEQ\label{eq:signal-dist}
u_i=
\left\{\BA{ll}
-1& \mbox{with probability }k/2n\\
+1& \mbox{with probability }k/2n\\
0&\mbox{otherwise,} \quad i=1,\ldots,n.
\EA\right.
\EEQ
The condition defining ${\mathcal U}$ in~\eqref{eq:u-set} can be written
\[
\max_{\substack{w\in\{-1,1\}^n\\Ax=0, \|x\|_1\leq 1}} \left\{u^T x + \xi \sum_{i=1}^n u_i w_i x_i \right\} \leq \xi
\]
and because the maximum is taken over a polyhedral set, this can be understood as
\[
\max_{x \in {\mathcal T}} u^T x \leq \xi
\]
where ${\mathcal T}\subset \reals^n$ is a finite set. When $u$ is distributed as in~\eqref{eq:signal-dist}, the left-hand side $\max_{x \in {\mathcal T}} u^T x$ of this last condition is a Rademacher process whose mean and fluctuations can be controlled, as detailed in the lemma below.

\begin{lemma}\label{lem:g-fluct}
Let $A\in\reals^{m \times n}$, with $\xi>0$ and $u$ distributed as in~\eqref{eq:signal-dist}, define
\BEQ\label{eq:nrm-ratio}
S(A)\equiv\max_{Ax=0} \frac{\|x\|_2}{\|x\|_1},
\EEQ
then the expected value of the max. can be bounded by
\[
M(A) \equiv \Expect\left[\displaystyle \max_{\substack{w\in\{-1,1\}^n\\Ax=0,\|x\|_1\leq 1}} \left\{u^Tx + \xi \sum_{i=1}^n u_i w_i x_i \right\}\right] \leq  (1+\xi) S(A) \Expect[\|u\|_2]
\]
and
\[
\Prob\left[\displaystyle \max_{\substack{w\in\{-1,1\}^n\\Ax=0, \|x\|_1\leq 1}} \left\{u^Tx + \xi \sum_{i=1}^n u_i w_i x_i \right\}\geq \xi\right] \leq 4 e^{-\frac{(\xi - M(A))^2}{4(1+\xi)^2 S^2(A)}}
\]
whenever $M(A)\leq \xi$.
\end{lemma}
\begin{proof}
First, remember that 
\[
\max_{\substack{Ax=0, \|x\|_1\leq 1\\ w\in\{-1,1\}^n}}~u^T(\idm+\xi \diag(w))x
\leq
(1+\xi)\max_{\substack{Ax=0\\ \|x\|_1\leq 1}}~|u|^T|x|
\]
where $|u|$ is the vector with components $|u_i|$ here. Then, a Cauchy-Schwarz inequality yields
\[
\Expect\left[\displaystyle\max_{\substack{Ax=0\\ \|x\|_1\leq 1}}~(1+\xi) |u|^T|x|\right]
\leq
\displaystyle (1+\xi)\max_{\substack{Ax=0\\ \|x\|_1\leq 1}}{\|x\|_2} ~\Expect\left[\|u\|_2\right].
\]
We then note that
\[
\max_{w\in\{-1,1\}^n}~\max_{\substack{Ax=0\\ \|x\|_1\leq 1}} \|(\idm+\xi \diag(w))x\|_2= (1+\xi)
\max_{\substack{Ax=0\\ \|x\|_1\leq 1}} \|x\|_2,
\]
when $\xi \in(0,1)$, and the concentration inequality follows from \citep[Cor.\,4.8]{Ledo05}.
\end{proof}

We summarize these last results in the following proposition, which highlights the role played by $S(A)$ in controlling the probability of recovering the signal $u$.
\begin{proposition}\label{prop:cond}
Suppose the signal $u$ is distributed as in~\eqref{eq:signal-dist}, $\beta>0$ and $A\in\reals^{m \times n}$ satisfies
\BEQ\label{eq:recov-cond}
S(A)<\frac{1}{\Expect[\|u\|_2]+2\beta+4\sqrt{\pi}}
\EEQ
then
\[
\Prob\left[u \notin {\mathcal U}\right] \leq 4 e^{-\beta^2}
\]
where ${\mathcal U}$ is the set of $\ell_1$-recoverable signals defined in~\eqref{eq:u-set}.
\end{proposition}
\begin{proof}
If $S(A)<{1}/{(\Expect[\|u\|_2]+2\beta+4\sqrt{\pi})}$, then there is a $\xi\in(0,1)$ such that $M(A)\leq \xi$ and
\[
(1+\xi)S(A)(\Expect[\|u\|_2]+2\beta+4\sqrt{\pi})<\xi,
\]
Lemma~\ref{lem:g-fluct} then yields the desired result.
\end{proof}

Because $\Expect[\|u\|_2] < \sqrt{k}$ (by Jensen's inequality), recovery with high probability can be obtained at slightly higher cardinalities $k$ than those required for recovery of all signals. Of course, this discrepancy vanishes if the random model for $u$ has uniformly distributed support of size exactly $k$. Here however, other choices of norm in~\eqref{eq:univ-cond} might produce different results.

\section{Tractable bounds}\label{s:tract}
In this section, we discuss methods to efficiently bound the ratio $S(A)$, i.e. control the Banach-Mazur distance of $\ell_1$ and $\ell_2$ on the nullspace of $A$. 

\subsection{Semidefinite relaxation} \label{ss:sdp-relax}
We now show how to compute tractable bounds on the ratio
\[
S(A)=\max_{Ax=0} \frac{\|x\|_2}{\|x\|_1},
\]
defined in~\eqref{eq:nrm-ratio}. This question is directly connected to the problem of efficiently testing Kashin decompositions (see \citep[\S4.1]{Szar10} for a discussion). We first formulate a semidefinite relaxation of this problem.
\begin{lemma}\label{eq:sdp-relax}
Let $A\in\reals^{m \times n}$,
\BEQ\label{eq:sdp}
S(A)^2\leq SDP(A) \equiv  \max_{\substack{\Tr(A^TAX)=0\\ \|X\|_1\leq 1,\, X\succeq 0}}~ \Tr X
\EEQ
where $SDP(A)$ is computed by solving a semidefinite program in the variable $X\in\symm_n$.
\end{lemma}
\begin{proof}
Writing $X=xx^T$, we have
\[
S(A)^2= \max_{\substack{\Tr(A^TAX)=0,\, \|X\|_1^2\leq 1,\\\Rank(X)=1,\, X\succeq 0}}~ \Tr X
\]
and dropping the rank constraint yields the desired result.
\end{proof}

We now connect the value of $S(A)$ with that of the function $\alpha_1(A)$ defined in \citep{Judi08,dAsp08a} as
\BEQ\label{alpha1}
\alpha_1(A) \equiv \max_{Ax=0} \frac{\|x\|_\infty}{\|x\|_1},
\EEQ
which can be computed by solving either a linear program \citep{Judi08} or a semidefinite program \citep{dAsp08a}. The following lemma bounds $S(A)$ using $\alpha_1(A)$.

\begin{lemma}\label{eq:approx-sdp-alpha}
Let $A\in\reals^{m \times n}$, we have
\[
\alpha_1(A) \leq S(A) \leq \sqrt{SDP(A)} \leq \sqrt{\alpha_1(A)}
\]
\end{lemma}
\begin{proof}
The first inequality simply follows from $\|x\|_ \infty \leq \|x\|_2$, the second from Lemma~\ref{eq:sdp-relax}. If $X$ solves~\eqref{eq:sdp}, $\Tr(A^TAX)=0$ implies $AX=0$, which means that the columns of $X$ are in the nullspace of~$A$. By definition of $\alpha_1(A)$, we then have $X_{ii} = \|X_{i}\|_\infty \leq \alpha_1(A) \|X_i\|_1$, hence $\Tr(X) \leq \alpha_1(A) \|X\|_1\leq  \alpha_1(A)$, which yields the desired result.
\end{proof}

The following proposition shows that if a matrix allows recovery of all signals of cardinality less than $k^*$, then the SDP relaxation above will efficiently certify recovery of all signals up to cardinality $O({k^*/\sqrt{n}})$. This is a direct extension of Lemma~\ref{eq:approx-sdp-alpha} and Proposition~\ref{prop:card-S}.

\begin{proposition}\label{prop:perf}
Suppose $A\in\reals^{m \times n}$ satisfies condition~\eqref{eq:diam} for some $S>0$, the semidefinite relaxation will satisfy
\BEQ\label{eq:ratio-bnd-perf}
S(A) \leq \sqrt{SDP(A)} \leq S^{-\frac{1}{4}}
\EEQ
and the semidefinite relaxation will certify exact decoding of all signals of cardinality at most $\sqrt{S}$.
\end{proposition}
\begin{proof} From Lemma~\ref{eq:approx-sdp-alpha}, we know that $\alpha_1\leq S(A)$ hence $\sqrt{SDP(A)}\leq \sqrt{S(A)}$. We conclude using Proposition~\ref{prop:card-S}.
\end{proof}

We can produce a second proof of this last result, which uses the norm ratio in~\eqref{eq:nrm-ratio} directly.
\begin{proposition}\label{prop:perf2}
Suppose $A\in\reals^{m \times n}$ satisfies condition~\eqref{eq:diam} for some $S>0$, the semidefinite relaxation will satisfy
\BEQ\label{eq:ratio-bnd-perf2}
S(A) \leq \sqrt{SDP(A)} \leq S^{-\frac{1}{4}}
\EEQ
and the semidefinite relaxation will certify exact decoding of all signals of cardinality at most $\sqrt{S}$.
\end{proposition}
\begin{proof} If $X$ solves the SDP relaxation in~\eqref{eq:sdp}, with $S(A)=S$ in~\eqref{eq:nrm-ratio}, then the rows of $X$ are in the nullspace of $A$, and satisfy $\|X_i\|_2\leq\|X_i|/\sqrt{S}$. Then, with $\|X\|_1$, 
\[
\Tr X \leq \sum_{i=1}^n \|X_i\|_\infty \leq \sum_{i=1}^n \|X_i\|_2 \leq \frac{\|X\|_1}{\sqrt{S}} \leq \frac{1}{\sqrt{S}}
\]
hence the desired result.
\end{proof}

Note that we are not directly using $X\succeq 0$ in this last proof, so the semidefinite relaxation can be replaced by a linear programming bound
\BEQ\label{eq:lp-bound}
\BA{rll}
LP(A)\equiv &\mbox{max.} & \Tr X\\
& \mbox{s.t.} & AX=0\\
& & \|X\|_1 \leq 1
\EA\EEQ
We now show that the $S^{-1/4}$ bound is typically the best we can hope for from the relaxation in~\eqref{eq:sdp}.
\begin{proposition}\label{prop:perf-limit}
Suppose $A\in\reals^{m \times n}$ with $n=2m$, then
\BEQ\label{eq:ratio-limit}
\frac{1}{\sqrt{2n}} \leq SDP(A)
\EEQ
and the semidefinite relaxation will certify exact decoding of all signals of cardinality at most $O(\sqrt{m})$.
\end{proposition}
\begin{proof} Let $Q$ be the orthoprojector on the nullspace of $A$. We have $Q\succeq 0$, $\Tr(Q)=m$, $\|Q\|_F=\sqrt{m}$ and $\|Q\|_1\leq \sqrt{n^2} \|Q\|_F \leq  n \sqrt{m}$, which means that $X=Q/(n\sqrt{m})$ is a feasible point of the SDP relaxation in~\eqref{eq:sdp} with $\Tr X=\sqrt{m}/n=1/\sqrt{2n}$ which yields the required bound on the optimal value of~\eqref{eq:sdp}.
\end{proof}

This means that if the matrix $A$ allows exact recovery of signals with up to (an unknown number) $S$ nonzero coefficients, then our relaxation will only exact certify recovery of signals with cardinality $O(\sqrt{S})$. The fact that approximating the recovery threshold $S$ is hard is not entirely surprising, $S$ in \eqref{eq:diam} is the Euclidean diameter of the centrally symmetric polytope $\{x\in \reals^n:Ax=0,\,\|x\|_1\leq1\}$. Computing the radius of convex polytopes is NP-Complete \citep{Freu85,Lova92,Grit93,Brie01}. In particular, \cite{Lova92} show that if we only have access to an oracle for $K$, then there is no randomized polynomial time algorithm to compute the diameter of a convex body $K$ within a factor $n^{1/4}$. In that sense, the approximation ratio obtained above is optimal. Here of course, we have some additional structural information on the set $K$ (it is a section of the $\ell_1$ ball) so there is a possibility that this bound could be improved. On the other hand, in the next section, we will see that if we are willing to add a few random experiments to $A$, then the diameter {\em can} be bounded with high probability by a randomized polynomial time algorithm.

\section{Geometric bounds} \label{s:geom}
Proposition \ref{prop:card-S} establishes a link between the sparse recovery threshold $S$ of a matrix $A$ and the diameter of the polytope $\{x\in \reals^n:Ax=0,\,\|x\|_1\leq1\}$. In this section, we first recall some classical results of geometric functional analysis and use these to quantify the sparse recovery thresholds of arbitrary matrices~$A$.

\subsection{Dvoretzky's theorem}
We first recall some concentration results on the sphere as well as classical results in geometric functional analysis  which control, in particular, the diameter of {\em random} sections of the $\ell_1$ ball (i.e. where $A$ is chosen randomly). Let~$\sigma$ be the unique rotation invariant probability measure on the unit sphere $\sphere^{n-1}$ of $\reals^n$, and $\|\cdot\|_K$ be a norm on $\reals^n$ with unit ball $K$, then
\BEQ\label{eq:conc-sph}
\sigma\left\{ x\in \sphere^{n-1}:\, |\|x\|-M(K)|\geq t M(K)\right\} \leq e^{-k(K) t^2}
\EEQ
with
\BEQ\label{eq:dv-dim}
k(K)=c n \left(\frac{M(K)}{b(K)}\right)^2
\EEQ
where $c>0$ is a universal constant, and
\BEQ\label{eq:m-b}
M(K)=\int_{\sphere^{n-1}} \|x\|d\sigma(x) \quad \mbox{and} \quad b(K)=\sup_{x\in\sphere^{n-1}} \|x\|.
\EEQ
\citet{Klar07} call $k(K)$ the {\em Dvoretzky dimension} of the convex set $K$. Part of the proof of Dvoretzky's theorem states that random sections of $K$ with dimension $k=k(K)$ are approximately spherical with high probability (w.r.t. the uniform measure on the Grassman ${\mathcal G}_{n,k}$). We write $B_p^n$ the $\ell_p$ ball of $\reals^n$. 

\begin{theorem}[\bf General Dvoretzky]
Let $E\subset\reals^n$ be a subspace of dimension $l\leq k(K)$ defined in~\eqref{eq:dv-dim}, chosen uniformly at random w.r.t. to the Haar measure on ${\mathcal G}_{n,k}$, then 
\[
\frac{c_1}{M(K)}(B_2^n \cap E) \subset (K \cap E) \subset \frac{c_2}{M(K)}(B_2^n \cap E)
\]
with probability $1-e^{-c_3 l}$, where $c_1,c_2,c_3>0$ are absolute constants.
\end{theorem}
\begin{proof}
See \citep[\S4]{Milm86} or \citep[Th.\,6.4]{Vers11} for example.
\end{proof}

This result means that random sections of convex bodies with dimension $k$ are approximately spherical with high probability. \citet{Milm97} show that the threshold $k(K)$ is sharp in the sense that random sections of dimension greater than $k(K)$ are typically not spherical. Because projections of sphere are spheres, there is thus a phase transition at $k(K)$: random sections of $K$ become increasingly spherical until they reach dimension $k(K)$ below which they are approximately spherical with high probability.

The diameter follows this phase transition as well, and the following result characterizes its behavior as the dimension of the subspace decreases (we write $K^*$ the polar of $K$).

\begin{theorem}[\bf Low $\mathbf{M^*}$ estimate] \label{eq:low-m*}
Let $E\subset\reals^n$ be a subspace of codimension $k$ chosen uniformly at random w.r.t. to the Haar measure on ${\mathcal G}_{n,n-k}$, then 
\[
\diam(K \cap E) \leq c\sqrt{\frac{n}{k}}M(K^*)
\]
with probability $1-e^{-k}$, where $c$ is an absolute constant.
\end{theorem}
\begin{proof}
See \citep{Pajo86} for example.
\end{proof}

The value of $M(K^*)$ is known for many convex bodies, including $l_p$ balls. In particular, $(B_1^n)^*=B_\infty^n$ and $M(B_\infty^n)\sim\sqrt{\log n /n}$ asymptotically. This means that random sections of the $\ell_1$ ball with dimension $n-k$ have diameter bounded by
\[
\diam(B_1^n \cap E)\leq c \sqrt{\frac{\log n}{k}}
\]
with high probability, where $c$ is an absolute constant (a more precise analysis allows the $\log$ term to be replaced by $log(n/k)$). 

\begin{theorem}[\bf Low $\mathbf{M}$ estimate] \label{eq:low-m}
Let $\lambda\in(0,1)$ and $k=\lfloor \lambda n \rfloor$ and $E\subset\reals^n$ be a subspace of codimension $k$ chosen uniformly at random w.r.t. to the Haar measure on ${\mathcal G}_{n,n-k}$, suppose $B^n_2 \subset K$ and
\[
M(K) \geq \sqrt{\lambda}
\]
then 
\[
\diam(K \cap E) \leq \frac{c \sqrt{1-\lambda}}{M(K)-\sqrt{\lambda}}
\]
with probability $1-c_2e^{-c_3\delta^2(1-\lambda)n}$, where 
\[
\delta=\frac{M^2(K)-\lambda}{1-M^2(K)}
\]
and $c_1,c_2,c_3$ are absolute constants.
\end{theorem}
\begin{proof}
See \citep[Th.B]{Gian05}.
\end{proof}

Note that the condition $B^n_2 \subset K$ means the set $K$ needs to be normalized by $b(K)$. \citet{Klar04} recently produced a similar result using $M(K)$ together with volume ratios. This result applies to all values of $M(K)/b(K)$, unfortunately, the dependence on $k$ is exponential instead of being polynomial.

\subsection{Connection with sparse recovery}
We have seen in Proposition \ref{prop:card-S} that the sparse recovery threshold associated with the $m$ linear observations stored in $A\in\reals^{m \times n}$, i.e. the largest signal cardinality for which all signals $u$ can be recovered exactly by solving the $\ell_1$-minimization problem in~\eqref{eq:l1-dec}, is given by the radius (or diameter) of the centrally symmetric convex polytope $\{x\in \reals^n:Ax=0,\,\|x\|_1\leq1\}$. By homogeneity, this is equivalent to producing lower bounds on $\|Fy\|_1$ over $\sphere^{n-m-1}$, the unit sphere of $\reals^{n-m}$. 

Proposition \ref{prop:card-S} (or \cite{Kash07}) shows that the sparse recovery threshold $S$ of the observations $A\in\reals^{m\times n}$ satisfies
\BEQ\label{eq:S-diam}
S\geq \frac{1}{\diam(\{x\in \reals^n:Ax=0,\,\|x\|_1\leq1\})^2} 
\EEQ
The low $M^*$ estimate in Proposition~\ref{eq:low-m*} together with the fact that $M(B_\infty^n)\sim\sqrt{\log n /n}$ then shows that choosing $m$ linear samples $A\in\reals^{m\times n}$ uniformly at random in the Grassman will allow us, with high probability, to recover all signals with at most $\frac{m}{c\log n}$ nonzero coefficients, by solving the $\ell_1$ minimization problem in~\eqref{eq:l1-dec} (again, the log term can be replaced by $\log(n/k)$).

\subsection{Approximating the diameter}
As we have seen above, finding good compressed sensing experiments means finding matrices $A\in\reals^{m \times n}$ for which $\|Fy\|_1$ is almost spherical, where $F$ is any basis for the nullspace. Bad matrices are matrices for which the norm ball of $\|Fy\|_1$ is much closer to a cross-polytope. This section is thus focused on measuring how spherical $\|Fy\|_1$ actually is. The key difficulty in high dimensions is that all centrally symmetric convex bodies look like spheres, except for a few ``spikes'' (or tentacles in \citet{Vers11}) with negligible volume, hence precisely characterizing the diameter using only probabilistic arguments is delicate.

If we notice that $\|Fy\|_1$ defines a norm on $\reals^{n-m}$, we can try to apply Dvoretzky's result in the normed space $(\reals^{n-m},\|Fy\|_1)$ instead of $(\reals^{n},\|x\|_1)$. The Dvoretzky dimension $k(K)$ would then act as an indirect measure of how Euclidean $\|Fy\|_1$ is. In compressed sensing terms, $k(K)$ computed in $(\reals^{n-m},\|Fy\|_1)$ will measure how many random experiments need to be added to the matrix $A$ so that all signals of size $O^*(n)$ can be recovered exactly by solving the $\ell_1$-minimization problem in~\eqref{eq:l1-dec}. The low $M$ estimate makes this statement even more explicit: Theorem~\ref{eq:low-m} directly links the ratio $M(K)/b(K)$ and the number $(1-\lambda)n$ of random experiments that need to be added to reach recovery threshold $S$ (through the diameter).

\subsubsection{Approximating the Dvoretzky dimension} 
We will see below that the quantities $M(K)$ and $b(K)$ which characterize the phase transition for sections of the norm ball of $\|Fy\|_1$ can be approximated efficiently. We first recall a result which can be traced back at least to \citep{Nest98a,Stei05}, approximating the mixed $\|\cdot\|_{2 \rightarrow 1}$ operator norm by a MAXCUT type relaxation.

\begin{proposition}\label{prop:12-norm}
Let $F\in\reals^{n \times n-m}$, then
\BEQ\label{eq:l12-approx}
\frac{2}{\pi} SDP(F) \leq \max_{\|x\|_2\leq 1} \|Fx\|_1^2 \leq SDP(F)
\EEQ
where
\BEQ\label{eq:sdp-bk}
\BA{rll}
SDP(F)=& \mbox{max.} & \Tr(XFF^T)\\
& \mbox{s.t.} & \diag(X)=\ones\\
&& X \succeq 0.
\EA
\EEQ
\end{proposition}
\begin{proof}
We can write
\[
\max_{\|x\|_2\leq 1} \|Fx\|_1^2 = \max_{\|u\|_\infty\leq 1} u^TFF^Tu
\]
and by convexity of $u^TFF^Tu$ this is equal to
\[
\max_{u\in\{-1,1\}^n} u^TFF^Tu
\]
and \citet{Nest98b} (using again the fact that $FF^T$ is positive semidefinite) shows that this problem can be approximated within a factor $2/\pi$ by the semidefinite relaxation in~\eqref{eq:sdp-bk}.
\end{proof}

This means that the mixed norm $b(K)$, which is typically hard to bound in probabilistic arguments, is approximated within a factor $2/\pi$ by solving a MAXCUT semidefinite relaxation when the norm ball is a section of the $\ell_1$ ball. We now recall a classical result showing that the spherical average $M(K)$ can be approximated by a Gaussian average.

\begin{lemma}\label{lem:m-gauss}
Let $f$ be a homogeneous function on $\reals^n$, then
\[
\int_{\sphere^{n-1}} f(x)d\sigma(x)=\left(\frac{1}{\sqrt{n}} +\frac{1}{4n^{3/2}}+o(n^{-3/2})\right)\Expect[f(g)]
\]
where $\sigma$ is the Haar measure on the sphere and $g\sim{\mathcal N}(0,\idm_n)$.
\end{lemma}
\begin{proof}
Because the Gaussian measure $\gamma$ is invariant by rotation, uniqueness of the Haar measure on $\sphere^{n-1}$ means that
\[
\int_{\sphere^{n-1}} f(x)d\sigma(x)=\lambda_n\int_{\reals^{n}} \|x\|_2 f(x/\|x\|_2)d\gamma(x)=\lambda_n\int_{\reals^{n}}  f(x)d\gamma(x)
\]
for some constant $\lambda_n$ satisfying
\[
\lambda_n=\int_{\reals^{n}} \|x\|_2 d\gamma(x)
\]
and we conclude using 
\[
\int_{\reals^{n}} \|x\|_2 d\gamma(x) = \frac{\sqrt{2}\Gamma((n+1)/2)}{\Gamma(n/2)}=\sqrt{n} -\frac{1}{4\sqrt{n}}+o(n^{-1/2})
\]
as $n$ goes to infinity.
\end{proof}

We can now easily compute $M(K)$, when $K$ is the unit ball of $\|Fy\|_1$, with
\BEQ\label{eq:m-l1}
M(K)=\left(\frac{1}{\sqrt{n}} +\frac{1}{4n^{3/2}}+o(n^{-3/2})\right)\sqrt{\frac{2}{\pi}}\sum_{i=1}^n\|F_i\|_2
\EEQ
where $F_i$ are the rows of the matrix $F$, with $F\in\reals^{n\times n-m}$ satisfying $AF=0$. The key difficulty with these approximations of the Dvoretzky dimension is that $M(B_1^n)$ is roughly equal to $\sqrt{2n/\pi}$, so the ratio $M(K)/b(K)$ is already constant and the $2/\pi$ approximation ratio for $b(K)$ only produces trivial bounds. Hence, even though we can expect matrices with high approximate ratio $M(K)/SDP(F)$ to be good sensing matrices, there are no guarantees that all such matrices will have high approximate ratios. 


\subsubsection{Approximating $M^*(K)$} 
We can also use the low $M^*$ bound in Theorem~\ref{eq:low-m*} to produce bounds on the diameter. Once again, the idea here is to apply this bound in the normed space $(\reals^{n-m},\|Fy\|_1)$ instead of $(\reals^{n},\|x\|_1)$, i.e. measure how many random experiments need to be added to the matrix $A$ so that all signals of size $S$ can be recovered exactly by solving the $\ell_1$-minimization problem in~\eqref{eq:l1-dec}. Solving for the dual norm is a convex problem, hence we can simply approximate $M^*$ by simulation. In the particular case of $(\reals^{n-m},\|Fy\|_1)$, this means computing
\BEQ\label{eq:l1-m*}
\Expect\left[\max_{\|Fy\|_1\leq 1} y^Tg \right] = \Expect\left[\min_{F^Tx=g} \|x\|_\infty \right] =\Expect\left[\min_{F^Tx=0} \|Fg+x\|_\infty \right]
\EEQ
by duality, where $g\sim{\mathcal N}(0,\idm_{n-m})$ (and assuming $F^TF=\idm_{n-m}$). Sampling both terms simply means solving one linear program per sample. Also, a simple Cauchy inequality shows that $M(K^*)$ is bounded above by $O(1/\sqrt{S})$. Since the target precision for our estimate of $M(K^*)$ is always larger than $1/\sqrt{n}$, this produces a recipe for a randomized polynomial time algorithm for estimating~$S$. In fact, following \citep{Bour88,Gian97,Gian05}, if $K\subset\reals^n$ is a symmetric convex body, $0<\delta,\beta<1$ and we pick $N$ points $x_i$ uniformly at random on the sphere $\sphere^{n-1}$ with
\[
N=\frac{c \log(2/\beta)}{\delta^2}+1
\]
where $c$ is an absolute constant, then
\[
\left|M(K^*)-\frac{1}{N}\sum_{i=1}^N \|x_i\|_{K^*}\right| \leq \delta M(K^*)
\]
with probability $1-\beta$.

\small{\bibliographystyle{plainnat}\bibsep 1ex
\bibliography{MainPerso}}
\end{document}

%% file: defs.tex
\newtheorem{theorem}{Theorem}[section]
\newtheorem{proposition}[theorem]{Proposition}

\newtheorem{lemma}[theorem]{Lemma}
\newtheorem{corollary}[theorem]{Corollary}
\renewenvironment{proof}{\textbf{Proof.}}{\QED\bigskip}

\newcommand{\BEAS}{\begin{eqnarray*}}
\newcommand{\EEAS}{\end{eqnarray*}}
\newcommand{\BEA}{\begin{eqnarray}}
\newcommand{\EEA}{\end{eqnarray}}
\newcommand{\BEQ}{\begin{equation}}
\newcommand{\EEQ}{\end{equation}}
\newcommand{\BIT}{\begin{itemize}}
\newcommand{\EIT}{\end{itemize}}
\newcommand{\BNUM}{\begin{enumerate}}
\newcommand{\ENUM}{\end{enumerate}}

\newcommand{\BA}{\begin{array}}
\newcommand{\EA}{\end{array}}

\newcommand{\ones}{\mathbf 1}

\newcommand{\reals}{{\mbox{\bf R}}}

\newcommand{\symm}{{\mbox{\bf S}}}  

\newcommand{\diam}{\mathop{\bf diam}}
\newcommand{\sphere}{{\mathbb S}}

\newcommand{\Rank}{\mathop{\bf Rank}}

\newcommand{\Card}{\mathop{\bf Card}}
\newcommand{\Tr}{\mathop{\bf Tr}}
\newcommand{\diag}{\mathop{\bf diag}}

\newcommand{\idm}{\mathbf{I}}

\newcommand{\Expect}{\textstyle\mathop{\bf E}}

\newcommand{\Prob}{\mathop{\bf Prob}}

\newcommand{\QED}{~~\rule[-1pt]{6pt}{6pt}}

